\documentclass[a4paper,10pt]{amsart}
\usepackage[latin1]{inputenc}
\usepackage[english]{babel}
\usepackage{comment}
\usepackage{amsmath,amsthm,amsfonts,amssymb,amstext,mathrsfs} %math
\usepackage{hyperref} %automatische links in pdfs
\usepackage{mathabx}
\usepackage{cite}
\usepackage{caption}
\usepackage{subfig}
\usepackage{accents}
\usepackage{graphicx}
\newtheorem{theorem}{Theorem}[section]

\newtheorem{lemma}[theorem]{Lemma}
\newtheorem{proposition}[theorem]{Proposition}
\newtheorem{corollary}[theorem]{Corollary}
\theoremstyle{definition}

\newtheorem{definition}[theorem]{Definition}
\theoremstyle{remark}
\newtheorem{remark}[theorem]{Remark}

\newtheorem{example}[theorem]{Example}

\makeatletter
\let\S\@undefined
\makeatother

\DeclareMathOperator{\rank}{rank}

\DeclareMathOperator{\conv}{conv}

\DeclareMathOperator{\Dec}{Dec}

\newcommand{\Mw}{M_{\omega}}
\newcommand{\R}{\mathbb{R}}
\newcommand{\BG}{\mathcal{G}}
\newcommand{\AGerz}{\langle A\rangle_{\mathcal{G}}}
\newcommand{\AGGerz}{\langle A\rangle_{\mathcal{G'}}}
\newcommand{\CGerz}{\langle C\rangle_{\mathcal{G}}}
\newcommand{\CGGerz}{\langle C\rangle_{\mathcal{G'}}}
\newcommand{\Aerz}{\langle A\rangle}

\newcommand{\Cerz}{\langle C\rangle}
\newcommand{\BGG}{\mathcal{G'}}
\newcommand{\Gmin}{\mathcal{G}_{min}}
\newcommand{\Gmax}{\mathcal{G}_{max}}

\usepackage{tikz}
\usetikzlibrary{matrix,arrows}
\usetikzlibrary{arrows,shapes,automata,backgrounds,petri}
\tikzstyle{knoten}=[fill,shape=circle,inner sep=2pt,outer sep=0pt,minimum size=2pt]
\tikzstyle{knoten2}=[fill,shape=circle,inner sep=2pt,outer sep=0pt,minimum size=2pt]
\tikzstyle{background}=[rectangle,
                                                fill=gray!10,
                                                inner sep=0.2cm,
                                                rounded corners=5mm]
\title{Bergman fans and decomposition complexes}
\author{Martin Dlugosch}
\address{Fachbereich Mathematik und Informatik, Universit\"at Bremen,
  Bibliothekstra\ss{}e 1, 28359 Bremen, Bremen, Germany.}
\email[Martin Dlugosch]{mdlug@math.uni-bremen.de }

\begin{document}
\maketitle
\begin{abstract}
We introduce \emph{decomposition complexes} of posets, which generalize order complexes. 
The main advantage of our construction is that decomposition complexes are closed under taking products. 
Other special instances of this theory include nested set complexes as well as Bergman complexes. 
\end{abstract}

\section{Introduction}

Let $P$ be a finite poset. 
The order complex of $P$ is the abstract simplicial complex whose vertices are the elements of $P$ and simplices are 
chains \emph{i.e.} totally ordered subsets of $P$. 
They have proven to contain important information about the poset, see 
\cite{Fol66}, \cite{Gom88} or \cite{B82}. 
An abstract simplicial complex can be identified with its face poset. 
In the case of order complexes we obtain the set of non-empty chains in $P$ ordered by inclusion.  

Though easy to handle order complexes have one defect. 
As for all abstract simplicial complexes there is no product structure 
since, thinking in terms of realizations, even the product of two 1-simplices is a quadrangle, 
which is not simplicial any more. 
%The order complex of the product of posets is not the product of order complexes. 
At least, realizations of order complexes of products can be chosen such that they subdivide 
products of realizations of order complexes in terms of polyhedral complexes \cite{Z95}. 

Here \emph{decomposition complexes} come into play. 
They can be seen as generalizations of order complexes of posets, which 
are closed under taking direct products. 
To make this work decomposition complexes describe face posets of 
objects called 
\emph{polytopal pseudo-complexes} instead of simplicial complexes. 
As a rule of thumb: the more special the case, the nicer the properties of the decomposition complex. 

Chapter \ref{ch2} is rather short. It contains all new basic definitions. 

The following Chapter \ref{ch2b} introduces conditions under which the main objects have nice properties. 
Some of the proofs used are fairly technical. 
They can be skipped without missing key ideas of the paper. 

Chapter \ref{ch3} is devoted to the problem of finding realizations \emph{i.e.} sets of polytopes whose face posets 
equal the decomposition complexes. This is crucial in the sense that illustrating examples are presented in terms of 
such realizations. 

The most important (in some sense even \emph{characterizing}) issue of decomposition complexes behaving well 
under taking products is considered in Chapter \ref{ch4}. 

The last two chapters deal with special instances of decomposition complexes which had a considerable impact on the 
development of the theory. 
Chapter \ref{ch5} is about \emph{nested set complexes}, which were 
introduced by Feichtner/Kozlov \cite{FK04}. 
They are the combinatorial core of the De Concini/Procesi theory of \emph{wonderful models of subspace arrangements}. 
Since the introduction is rather short, we refer to \cite{FK04}, \cite{Fe05} or \cite{FM05} for more background. 

The other instance are \emph{Bergman fans} of matroids in Chapter \ref{ch6}. 
Bergman fans were introduced by Bergman \cite{BE71}, but they have received attention recently, after 
Sturmfels \cite{STU02} recognized them working in the emerging field of \emph{tropical geometry}. 
Helpful introductions to Bergman fans of matroids include \cite{AK04} and \cite{FS04}. 

Though Bergman fans and nested set complexes are defined in different languages, they can both be seen 
as special cases of decomposition complexes. 

\section{Decomposition sets}
\label{ch2}
Let $P$ be a finite poset.

\begin{definition}[Decompositions, Decomposition sets]
A \emph{decomposition} is a triple $( x,z,y )$ of elements of $P$ with $x < z<y$ 
such that $  [x,z]\times [z,y] \cong [x,y]$ via an isomorphism $\psi$ which sends $(u,z)\mapsto u$ and $(z,v)\mapsto (v)$. 
A decomposition is \emph{trivial} if either $z =x$ or $z=y$. 
Otherwise the decomposition is \emph{proper}. 

A \emph{decomposition set} $\BG$ is a set of decompositions in $P$ which 
contains all trivial decompositions. 
\end{definition}
The set of decomposition sets can be ordered by inclusion. 
There is always a minimal decomposition set $\Gmin$, consisting of the trivial decompositions, 
as well as a maximal decomposition set $\Gmax$, consisting of all decompositions. 

For a decomposition set $\BG$ and a subset $A\subseteq P$ we set 
\begin{displaymath}
\langle A\rangle_{\BG}:= \bigcap \{ A \subseteq B \subseteq P ~ | ~ 
\forall ~ (x,z,y)\in \BG ~ :~  x,y \in B
\Rightarrow z \in B \} 
\end{displaymath} 
\begin{remark}
\label{rem:closure}
It is easy to check that $\langle \cdot \rangle_\BG$ is a closure operator \emph{i.e.} 
\begin{itemize}
\item $A\subseteq \langle A\rangle_\BG$,
\item $A\subseteq B \Rightarrow \langle A\rangle_\BG \subseteq \langle B\rangle_\BG$,
\item $\langle ~ \langle A\rangle_\BG ~ \rangle_\BG = \langle A\rangle_\BG$.
\end{itemize}
\end{remark}

\begin{definition}[Decomposition complexes]
Let $P$ be a finite poset and $\BG$ a decomposition set of $P$. 
The \emph{decomposition complex} of $P$ resp. the decomposition set $\BG$ is defined as 
\begin{displaymath}
\mathcal{D}(P,\BG) := \{ \CGerz  ~| ~ C \text{ non-empty chain in } P\}. 
\end{displaymath}
\end{definition}

\begin{remark}
For any poset $P$ the decomposition complex resp. the minimal decomposition set 
is the face poset of the order complex of $P$, since the operator $\langle \cdot \rangle_{\Gmin}= id$. 
\end{remark}

\begin{example}
\label{ex:first}
Let $P$ be the power set lattice of rank $2$. 
It has the property that any triple $x<z<y$ gives a decomposition. 
Figure \ref{fig:first} shows Hasse-diagrams of $P$ on the left, its decomposition complex resp. the 
minimal decomposition set in the middle and its decomposition complex resp. the maximal decomposition set on the right.   
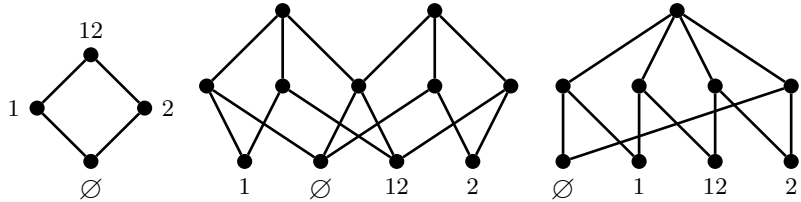
\begin{figure}[h]
{\small
\centering
\subfloat{
 \begin{tikzpicture}
[line width=1,description/.style={fill=white,inner sep=2pt},node distance=1cm]
  \node [knoten](1) [label=below:$\emptyset$]{};
  \node [knoten](2) [above left of=1,label=left:$1$] {};
  \node [knoten](3) [above right of=1,label=right:$2$]{};
  \node [knoten](4) [above right of=2,label=above:$12$]{};
  %\node [knoten](4) [above right of=2,label=right:4]{};
  %\node [knoten](12) [above right of=6,label=left:12]{};
  \path[-] (1) edge (2) edge (3);
  %\path[-] (6) edge (2) edge (3) edge (12);
  \path[-] (4) edge (3) edge (2);
%  \path[-] (oben) edge node[above right] {2}(rechts);
  %\path[-] (oben) edge node[above left] {1} (links);
  %\path[-] (links) edge node[below left] {3} (unten);
  %\path[-] (unten) edge node[below right] {4} (rechts);
\end{tikzpicture}
}
{
\begin{tikzpicture}[line width=1,scale=0.5]
[description/.style={fill=white,inner sep=2pt}]
  \node [knoten](leer) at (0,0) [label=below:$\emptyset$] {};
  \node [knoten](1) [left of=leer,label=below:$1$] {};
  \node [knoten](12) [right of=leer,label=below:$12$] {};
  \node [knoten](2) [right of=12,label=below:$2$] {};
  \node [knoten](leer1) at (-3,2) {};
  \node [knoten](112) [right of=leer1] {};
    \node [knoten] (leer12) [right of=112] {};
  \node [knoten](leer2) [right of=leer12] {};
\node [knoten](212) [right of=leer2] {};
  \node [knoten](all)  [above of=112] {};
  \node [knoten](all2) [above of=leer2] {};
  \path[-] (leer) edge (leer1) edge (leer2);
\path[-] (1) edge (leer1) edge (112);
\path[-] (2) edge (212) edge (leer2);
\path[-] (12) edge (112) edge (212);
\path[-] (all) edge (leer1) edge(112);
\path[-] (all2) edge (leer2) edge(212);
\path[-] (leer12) edge (all) edge(all2) edge(leer) edge (12);
\end{tikzpicture}
}
{
\begin{tikzpicture}[line width=1,scale=0.5]
[description/.style={fill=white,inner sep=2pt}]
  \node [knoten](leer) at (0,0) [label=below:$\emptyset$] {};
  \node [knoten](1) [right of=leer,label=below:$1$] {};
  \node [knoten](12) [right of=1,label=below:$12$] {};
  \node [knoten](2) [right of=12,label=below:$2$] {};
  \node [knoten](leer1) [above of=leer] {};
  \node [knoten](112) [right of=leer1] {};
  \node [knoten](212) [right of=112] {};
  \node [knoten](leer2) [right of=212] {};
  \node [knoten](all) at (3,4) {};
 % \node [knoten](1) [right of=leer,label=below:$1$] {};
  \path[-] (leer) edge (leer1) edge (leer2);
\path[-] (1) edge (leer1) edge (112);
\path[-] (2) edge (212) edge (leer2);
\path[-] (12) edge (112) edge (212);
\path[-] (all) edge (leer1) edge (leer2) edge(112) edge (212);
\end{tikzpicture}

}}
\caption{$P$ and its decomposition complexes resp. the minimal and the maximal decomposition set}
\label{fig:first}
\end{figure}
\end{example}

\begin{remark}
For a decomposition $(x,z,y)$ of $P$ we obtain a \emph{dual decomposition} $(y,z,x)$ in $P^{op}$. 
Denote the decomposition set consisting of the duals of $\BG$ by $\BG^{op}$. 
Since chains in $P$ are chains in $P^{op}$, too, we obtain that $\mathcal{D}(P,\BG)\cong \mathcal{D}(P^{op},\BG^{op})$. 
\end{remark}

\section{Properties of nice decomposition sets}
\label{ch2b}
Starting with a decomposition of an interval $[x,y]$ we obtain a decomposition of subintervals 
$[u,v]\subseteq [x,y]$ in the following way:

For $x\leq u,v \leq y$, $(u_1,u_2)=\psi^{-1}(u)$ and $(v_1,v_2)=\psi^{-1}(v)$, the triple  
$(u,\psi(v_1,u_2,v)$ is a decomposition of the interval $[u,v]$, because  
\begin{align*} 
[u,\psi(v_1,u_2)]\times [\psi(v_1,u_2),v]  
&\cong[(u_1,u_2),(v_1,u_2)]\times [(v_1,u_2),(v_1,v_2)] \\
&\cong [u_1,v_1]\times [u_2,v_2]  \\ 
&\cong[(u_1,u_2),(v_1,v_2)] . \\
 \end{align*}

In particular the isomorphisms $ [u,\psi(v_1,u_2)] \times [\psi(v_1,u_2),v ]\xrightarrow{\sim} [u,v]$ 
is just the restriction of $\psi$ to the subinterval $[(u_1,u_2),(v_1,v_2)]$ and its image $[u,v]$ under $\psi$. 
The situation is illustrated in Figure \ref{fig:lem}. 

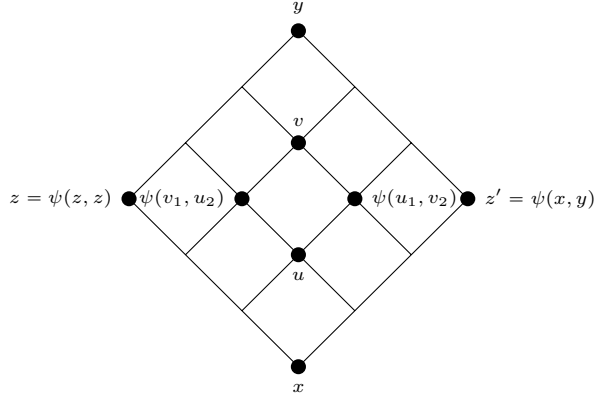
\begin{figure}[h]
\begin{tikzpicture}[rotate=45,scale=0.7]
\tikzstyle{every node}=[font=\scriptsize]
\draw (0,0) grid [step=1.5cm] (4.5,4.5);
\node[knoten,label= left:{$\psi(v_1,u_2)$}] at (1.5,3){};
\node[knoten,label= right:{$\psi(u_1,v_2)$}] at (3,1.5){};
\node[knoten,label=above:{$v$}] at (3,3) {};
\node[knoten,label=below:{$u$}] at (1.5,1.5) {};
\node[knoten,label=below:{$x$}] at (0,0) {};
\node[knoten,label=above:{$y$}] at (4.5,4.5) {};
\node[knoten,label= right:{$z'=\psi(x,y)$}] at (4.5,0){};
\node[knoten,label= left:{$z=\psi(z,z)$}] at (0,4.5){};
\end{tikzpicture}
\caption{A pattern of an excerpt of a Hasse-diagramm}
\label{fig:lem}
\end{figure}

We order the set of decompositions by defining that a decomposition is smaller than another
 if its isomorphism is the restriction of the other isomorphism by the above construction. 
In particular the decomposed interval is a subinterval of the other interval. 
It is easy to check that this generates an order relation on 
the set of decompositions of intervals in $P$. \label{2b}

\begin{definition}
A proper decomposition $(x,z,y)$ is said to be \emph{minimal with respect to} $A \subseteq P$ if it is minimal 
among all proper decompositions for which $x,y \in A$. 

Given a decomposition $(x,z,y)$, its \emph{complementary decomposition} is $(x,\psi(x,y),y)$. 
It is easy to see that a decomposition if proper if and only if its complementary decomposition if proper. 
A decomposition set $\BG$ is called \emph{symmetric} if for any decomposition its complementary decomposition 
is contained in $\BG$, too. 
\end{definition}
%Note that any decomposable the interval $ [x,y]$ can be decomposed by 
%\begin{displaymath}
%[x,y]\cong [x,z] \times [x,\psi(x,y)].
%\end{displaymath}
Of course the isomorphims for complementary decompositions are closely related. 
So it left to the reader to check, that minimality is symmetric in the sense that $(x,z,y)$ is minimal resp. 
$\AGerz$ if and only if 
$(x,\psi(x,y),y)$ is minimal resp. $\AGerz$. 

\begin{lemma}
\label{lem:min}
Let $\BG$ be a decomposition set, which is symmetric and downwards closed and 
$A$ be a subset with $\langle A \rangle_{\BG} =A$. 
Then a decomposition $(x,z,y)$ of $\BG$ is minimal with respect to $A$ if and only if 
$[x,y]\cap A \subseteq \{x,z,\psi(x,y),y\}$. 
\end{lemma}
\begin{proof}
Let $(x,z,y)$ be a decomposition, which is minimal with respect to $A$. 
Since the decomposition set $\BG$ is symmetric, $(x,\psi(x,y),y)$ belongs to $\BG$ as well. 
Thus $A$ contains both $z$ and $\psi(x,y)$, because it is closed under decompositions of $\BG$. 
Let $a$ be an element of $A\cap [x,y]$ and let $(u,v)$ be its image under $\psi$.
Then both the smaller decompositions 
$(x,\psi(u,x),a)$ and $(a,\psi(z,v),y)$ 
belong to $\BG$ as well, because it is downwards closed. 
Thus there is a contradiction to minimality unless those decompositions are trivial which means
\begin{displaymath}
(u,x) \in \{ \psi^{-1}(x), \psi^{-1}(z) \} \text{ and } 
(z,v) \in \{ \psi^{-1}(z), \psi^{-1}(y) \} .
\end{displaymath}
Thus both $u \in \{ x,z \}$ and $v \in \{ z,y \}$. 
Those four possibilities belong to the cases where $a$ equals $x,y,z$ or $\psi(x,y)$.

For the other implication, let $(x,z,y)$ be a decomposition such that $[x,y]\cap A \subseteq \{x,z,\psi(x,y),y\}$. 
Assume there is a lesser, proper decomposition $(u,z',v)$. 
Then in particular $u,v  \in [x,y]\cap A \subseteq \{x,z,\psi(x,y),y\}$. 
This lesser decomposition is contained in $\BG$, too, because $\BG$ is downwards closed. 
Since $A$ is closed under decompositions of $\BG$ we obtain $z'\in A$, which means, that 
$z'$ is also an element of $[x,y]\cap A$.
Since $u< z' < v$, we can conclude, that $u=x$ and $v=y$. 
\end{proof}

Let $\BG$ be symmetric and downwards closed. 
\begin{remark}
\label{rem:seq}
Consider the following construction. 
For $\AGerz  \subseteq P$ one can develop a finite sequence of decomposition sets 
\begin{displaymath}
\label{sequence}
\Gmin=\BG_0\subset \ldots \subset \BG_k=\BGG
\end{displaymath}
such that $\BG_{i+1}$ is the union of $\BG_{i}$ and the lower hull of a pair $(x,z,y),(x,\psi(x,y),y)$ of decompositions, 
which are minimal resp. $\langle A \rangle_{\BG_i}$. 
If there are no more decompositions of $\BG$, which are minimal resp. $\AGerz$, 
the sequence ends with a decomposition set~$\BGG$. 
\end{remark}
\begin{lemma}
\label{lem:bigenough}
$\BGG$ is big enough such that $\langle A \rangle_ \BGG = \AGerz$.
\end{lemma}
\begin{proof}
Assume $(x,z,y)$ is a minimal among $\BG \backslash \BGG$ such that $x,y \in \AGGerz$ but $z \notin \AGGerz$. 
This decomposition is not minimal resp. $\AGGerz$, otherwise the sequence could have been further extended. 
By Lemma \ref{lem:min}, there has to be an element $a \in \AGGerz$ such that $a \notin \{ x,z, \psi(x,y),y \}$
Let $(u,v)=\psi^{-1}(a)$. 
Thus at least one of the lesser decompositions $(x,\psi(u,x),a)$ and $(a,\psi(z,v),y)$ is a proper decomposition 
of a proper subinterval of $[x,y]$. 
By minimality choice of $(x,z,y)$ those lesser decompositions grant, that both $\psi(u,x)$ and $\psi(z,v)$ 
are contained in $\AGGerz$, too. 

There is a third decomposition $(\psi(u,x),z,\psi(z,v) )$ lesser than $(x,z,y)$. 
Again, by minimality of $(x,z,y)$ we obtain $z \in \AGGerz$, which is a contradiction.   
\end{proof}

\begin{proposition}
\label{valids}
For a symmetric, downwards closed decomposition set $\BG$ and $A\subseteq P$ the following statements are equivalent
\begin{itemize}
\item[(i)] $A \in \mathcal{D}(P,\BG),$ 
\item[(ii)] $A= \langle  C \rangle$ for any maximal chain $C\subseteq A$, 
\item[(iii)] For any pair of maximal chains $C,C'$there is a sequence $C=C_0,\ldots,C_k=C'$ such that 
$C_{i+1}=C_i  \backslash \{ z'\} \cup \{ z \}$ with 
$(x,z',y) \in \BG$ and $x,y\in \langle C_i \rangle$. 
\end{itemize}
\end{proposition}
\begin{proof}
(iii) $\Rightarrow$ (ii)
For any pair $C,C'$ of maximal chains in $A$, the chains in the sequence are constructed such that 
$C_{i+1}\subseteq \langle C_i\rangle_\BG$, because $A=\langle A \rangle_\BG$. 
By the properties of Remark \ref{rem:closure} we recursively obtain $C'\subseteq \langle C\rangle_\BG$. 
Thus all maximal chains of $A$ are contained in $\langle C\rangle_\BG$.   

(i) $\Rightarrow$ (iii)
W.l.o.g. the generating chain $C$ is maximal in $A$. 
Using the construction of Remark \ref{rem:seq}, let $\Gmin=\BG_0\subseteq \ldots\subseteq \BG_k= \BGG$ 
be a sequence of symmetric and downwards closed decomposition sets such that 
there is only one pair of decompositions $(x_i,z_i,y_i),(x_i,z_i',y_i)$ in $\BG_{i+1}\backslash \BG_{i}$, 
which are minimal respective $\langle C \rangle_{\BG_i}$ 
and $\langle C \rangle_{\BG_k}=\langle C \rangle_{\BG}$. 
Lemma \ref{lem:bigenough} guaranties, that the latter decomposition set $\BGG$ can be chosen 
big enough such that $\CGGerz = \CGerz=A$. 
By construction and Lemma $\ref{lem:min}$ we know 
$| \langle C \rangle_{\BG_{i+1}} \backslash \langle C \rangle_{\BG_i} | \leq 1$. 

Let $C'$ be a maximal chain in $A=\langle C \rangle_{\BGG}$. 
Observe that maximal chains in $A$ containing $x_i$ and $y_i$ contain either $z_i$ or $z'_i$, 
the unique connected components of this interval.   
Construct maximal chains $C_i$ in $A$ recursively by setting 
\begin{displaymath}
C_k:=C' \text{ and } C_{i}:=
\begin{cases}
C_{i+1} &\text{ if } z_i \in \langle C \rangle_{\BG_i} \\
C_{i+1}\backslash \{z'_i\} \cup \{z_i\}&\text{ if } z'_i \in \langle C \rangle_{\BG_i}.
\end{cases}
\end{displaymath}
Minimality of the used decompositions guarantee that we obtain chains again. 
Every exchange of $z'_i$ with $z_i$ belongs to the change of incomparable elements of a subinterval $[x_i,y_i]$of $A$, 
whose size is $4$. 
Thus those chains are maximal again, too.  
In the end of this sequence there is only one possible chain left, namely $C_0=C$.

(ii) $\Rightarrow$ (i)
This implication is obvious. 
\end{proof}

\begin{remark}
\label{partofa}
For any set $A=\Aerz$, one easily obtains an equivalence relation on the set of maximal chains from $(iii)$ by defining 
two maximal chains to be equivalent if there is such a sequence leading from one to the other. 
The union of the chains in a specific class satisfies the properties of Proposition \ref{valids}. 
\end{remark}

\begin{proposition}
A set $A$ satisfying the conditions of Proposition \ref{valids} for a decomposition set $\BG$ is a lattice. 
\end{proposition}
\begin{proof}
Let $A=\CGerz$. 
Again, we consider a sequence of decomposition sets, as in Remark \ref{rem:seq}, 
$\Gmin = \BG_0\subseteq \ldots \subseteq \BG_k = \BGG$ such that $\langle C \rangle_{\BGG}=\langle C \rangle_{\BG}$ and 
there are only two (complementary) decomposition in $\BG_{i+1}\backslash \BG_i$, 
which are minimal with respect to $\langle C \rangle_{\BG_i}$. 
Again, by this choice we obtain $|\langle C \rangle_{\BG_{i+1}} \backslash \langle C \rangle_{\BG_{i}}| \leq 1$. 

We will show that $\langle C\rangle_{\BGG}=\langle C\rangle_{\BG}$ is a lattice by induction on $i$. 
For the start of the induction we see that $\langle C \rangle_{\Gmin}= C$, which is a chain and thus a lattice in particular. 
Now assume, that $\langle C\rangle_{\BG_i}$ is a lattice. 
Let  $ \{ \langle C \rangle_{\BG_{i+1}} \backslash \langle C \rangle_{\BG_{i}} \}
= \{ z'_i \}$ and $u \in \langle C \rangle_{i+1}$. 
For showing that $z'_i \wedge u$ exists, note that by the induction assumption $x_i \wedge u$ exists and
$x_i \wedge u \leq z',u$. 
So $P_{\leq z'_i} \cap P_{\leq u}$ is not empty. 
Let $v,w$ be maximal elements of $P_{\leq z'_i} \cap P_{\leq u}$. 
Unless either $v=z'_i$ or $w=z'_i$, their join $v \vee w$, which exists by induction, has the property 
that $v \vee w \leq z'_i,u$, since both $v,w \leq z_i',u$. 
Thus $v \vee w=v=w=z'_i \wedge u$, because $v,w$ were chosen maximal. 
\end{proof}

\section{realizations}
\label{ch3}
\begin{definition}[Realizations of decomposition sets]
Let $\BG$ be a symmetric, downwards closed decomposition set of $P$. 
A $\BG$-realization is an embedding $\phi : P \rightarrow B_n$ 
of $P$ to the power set lattice of some finite set $[n]:=\{ 1, \ldots , n\}$ 
such that for all pairs of complementary decompositions $(x,z,y),(x,z',y)\in \BG$ the following holds: 
\begin{displaymath}
\phi(x)=\phi(z)\cap \phi(z') \quad \text{   as well as   } \quad
\phi(y)=\phi(z)\cup \phi(z').
\end{displaymath}
If there exists such an embedding $P$ is called $\BG$-realizable.   
$\phi$ is called the \emph{realization of the decomposition set}. 
\end{definition}

\begin{example}
Any poset is $\Gmin$- realizable by enumerating the elements of $P$ by $x_1,\ldots, x_n$ 
and setting 
$i\in \phi(x)$ if and only if $x_i \leq x$ in $P$.
\end{example}

\begin{definition}[Polytopal pseudo-complexes]
A \emph{polytopal pseudo-complex} $\mathcal{F}$ 
is a finite set of polytopes such that faces of polytopes of $\mathcal{F}$ are 
contained in $\mathcal{F}$ and intersections of polytopes of $\mathcal{F}$
are unions of polytopes of $\mathcal{F}$. 
\end{definition}
The only difference to the definition of \emph{polyhedral/polytopal complexes} \cite{Z95} is, that intersections are allowed to be 
unions of polytopes instead of just single polytopes. 

\label{incid}
For $A \in B_n$ let us denote its incidence vector $e_A \in \{0,1\}^n$ by 
$(e_A)_i=1$ if an only if $i \in A$. 
We assign a polytope to $A\subseteq P$ by setting 
\begin{displaymath}
\Gamma(A):= \conv(e_A)=\{ \sum_{x\in A} \lambda_x \cdot e_{\phi(x)} ~ | ~ \lambda_x > 0, \sum \lambda_x=1\}.
\end{displaymath}

\begin{definition}[Realizations of decomposition complexes]
Let $\mathcal{D}(P,\BG)$ be a decomposition complex with $\BG$-realization $\phi$. 
Then clearly, the set of polytopes
\begin{displaymath}
\widetilde{\mathcal{D}}(P,\BG):=~ \{ \Gamma(A) ~ | ~  A \in \mathcal{D}(P,\BG) \}
\end{displaymath} 
forms a polytopal pseudo-complex. 
We call it a \emph{realization of the decomposition complex} of $P$ respective $\BG$.
\end{definition}
Though $\Gamma(A)\cap\Gamma(B)=\Gamma(A\cap B)$ is always true, the statement 
$A\cap B\in \mathcal{D}(P,\BG)$ may not be true for $A,B\in \mathcal{D}(P,\BG)$, see Example \ref{ex:maynot}. 
In fact $\Gamma(A\cap B)$  is the union of polytopes of maximal faces in $A\cap B$, see Remark \ref{partofa}. 

\begin{example}
\label{ex:maynot}
Let $P$ be the poset below. 
Then, with respect to the maximal decomposition set, the closures under $\langle \cdot \rangle$ of the chains $\{0,a,c,d\}$ and $\{0,b,c,d\}$ coincide. 
The same is true for the chains $\{0,a,c,e\}$ and $\{0,b,c,e\}$. 
The intersection of those two closures is $\{ 0,a,b,c\}$. 
Though closed, it is not generated by any of its maximal chains. 
Choosing a realization of $\Gmax$ (for example via the construction of Proposition \ref{prop:atom}) gives an example 
of a realization of some decompositions complex, which is no polytopal complexes. 
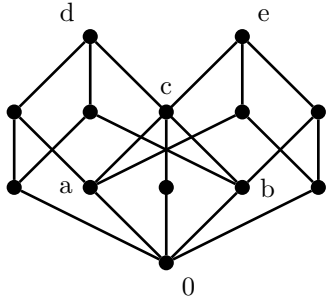
\begin{figure}[h]
\begin{tikzpicture}[line width=1,description/.style={fill=white,inner sep=2pt},node distance=1cm]
  \node [knoten](0) [label=below right:0]{};
  \node [knoten](d) [above of=0] {};
  \node [knoten](a) [left of=d,label=left:a]{};
  \node [knoten](b) [right of=d,label=right:b]{};
  \node [knoten](c) [above of=d,label=above:c]{};
  \node [knoten](e) [left of=c]{};
  \node [knoten](f) [left of=e]{};
  \node [knoten](i) [left of=a]{};
  \node [knoten](j) [right of=b]{};
  \node [knoten](g) [right of=c]{};
  \node [knoten](h) [right of=g]{};
  \node [knoten](1) [above of=e,label=above left:d]{};
  \node [knoten](2) [above of=g,label=above right:e]{};
  \path[-] (0) edge (a) edge (c) edge (b) edge (i) edge (j);
  \path[-] (1) edge (b) edge (e) edge (f);
  \path[-] (2) edge (g) edge (a) edge (h);
  \path[-] (e) edge (b);
  \path[-] (g) edge (a);
  \path[-] (i) edge (e) edge (f);
  \path[-] (j) edge (g) edge (h);
  \path[-] (b) edge (h);
  \path[-] (a) edge (f);
\end{tikzpicture}
\caption{A poset $P$}
\label{fig:pseudo}
\end{figure}
\end{example}

\begin{proposition}
Let $\BG$ be a decomposition set with realization $\phi$, which contains a pair of complementary decompositions 
$(x,z,y),(x,z',y)$. Then 
\begin{displaymath}
\Gamma(\{x,z,z',y\})  = \Gamma(\{x,z,y\})  \cup \Gamma(\{x,z',y\}).
\end{displaymath}
\end{proposition}
\begin{proof}
The second inclusion is trivial. 
For the first one, let 
$v=\lambda_1 e_{\phi(x)} + \lambda_2 e_{\phi(z)} + 
\lambda_3 e_{\phi(z')} + \lambda_4 e_{\phi(y)}$ be arbitrary inside $\Gamma(\{x,z,z',y\})$. 
W.l.o.g. we can assume $\lambda_2 \geq \lambda_3$. 
Since $\phi(z')=(\phi(y)\backslash \phi(z))\cup \phi(x)$, we can express 
$v$ as $(\lambda_1 + \lambda_3) e_{\phi(x)} + (\lambda_2 - \lambda_3) e_{\phi(z)} + 
(\lambda_1 +\lambda_3) e_{\phi(y)}\in \Gamma(\{x,z,y\})$. 
\end{proof}

\begin{corollary}
\label{cor:conedec}
Let $\BG$ be a realizable decomposition set and $A\subseteq P$. 
Then the polytope $\Gamma(A)$ is the union of polytopes $\Gamma(C)$ of maximal chains $C$ inside $A$. 
\end{corollary}

Note, that for decomposition sets $\BGG \subseteq \BG$, a $\BG-$realization $\psi$ is a $\BGG-$realization as well. 
\begin{corollary}
Since $\AGerz \subseteq \AGGerz$, the realization $\widetilde{\mathcal{D}}(P,\BGG)$ is a subdivision 
of the realization $\widetilde{\mathcal{D}}(P,\BG)$. 
\end{corollary}

\begin{example}
\label{ex:second}
Let $P$ be the power set lattice of rank $2$ as in Example \ref{ex:first}.  
By definition of $P$, the identity is a canonical $\Gmax$-realization. 
Figure \ref{fig:quadrate} shows the realizations $\mathcal{D}(P,\Gmin)$ and $\mathcal{D}(P,\Gmax)$.
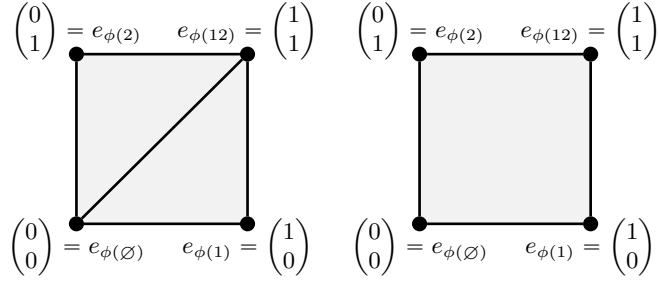
\begin{figure}[h]
{\small
\centering
\subfloat
{
\begin{tikzpicture}[line width=1,scale=0.75]
[description/.style={fill=white,inner sep=1pt}]
\filldraw[gray!10] (0,0) -- (3,0) -- (3,3) -- (0,3) -- (0,0);
  \node [knoten2](leer) at (0,0) 
[label={[label distance=-0.3cm]270:${\left(\begin{matrix}0\\ 0\end{matrix}\right)=e_{\phi(\emptyset)}}$}]{};
  \node [knoten2](2) at (0,3) 
[label={[label distance=-0.3cm]90:${\left(\begin{matrix}0\\ 1\end{matrix}\right)=e_{\phi(2)}}$}]{};
  \node [knoten2](1) at (3,0) 
[label={[label distance=-0.3cm]270:${e_{\phi(1)}=\left(\begin{matrix}1 \\ 0\end{matrix}\right)}$}]{};
  \node [knoten2](12) at (3,3) 
[label={[label distance=-0.3cm]90:${e_{\phi(12)}=\left(\begin{matrix}1\\ 1\end{matrix}\right)}$}]{};
\path[-] (leer) edge (1) edge(2) edge (12);
\path[-] (12) edge (1) edge(2);
\end{tikzpicture}
}
{
\begin{tikzpicture}[line width=1,scale=0.75]
[description/.style={fill=white,inner sep=1pt}]
\filldraw[gray!10] (0,0) -- (3,0) -- (3,3) -- (0,3) -- (0,0);
  \node [knoten2](leer) at (0,0) 
[label={[label distance=-0.3cm]270:${\left(\begin{matrix}0\\ 0\end{matrix}\right)=e_{\phi(\emptyset)}}$}]{};
  \node [knoten2](2) at (0,3) 
[label={[label distance=-0.3cm]90:${\left(\begin{matrix}0\\ 1\end{matrix}\right)=e_{\phi(2)}}$}]{};
  \node [knoten2](1) at (3,0) 
[label={[label distance=-0.3cm]270:${e_{\phi(1)}=\left(\begin{matrix}1 \\ 0\end{matrix}\right)}$}]{};
  \node [knoten2](12) at (3,3) 
[label={[label distance=-0.3cm]90:${e_{\phi(12)}=\left(\begin{matrix}1\\ 1\end{matrix}\right)}$}]{};
\path[-] (leer) edge (1) edge(2);
\path[-] (12) edge (1) edge(2);
\end{tikzpicture}
}}
\caption{Realizations of the power set lattice of rank $2$ resp. the minimal (left) and the maximal 
decomposition set(right)}
\label{fig:quadrate}
\end{figure}
\end{example}

\begin{example}
We close this section with an example of a poset, for which the maximal decomposition set is not realizable. 
Let $P$ be the face poset of a quadrangle augmented by a minimal element. 
Then the decomposition poset is not realizable, since the decomposition complex is not even graded. 
Thus it cannot even be the face poset of some regular CW-complex. 
\end{example}

\section{Products}
\label{ch4}
\begin{proposition}
For decomposition sets $\BG_1$ of $P_1$ and $\BG_2$ of $P_2$, 
\begin{displaymath}
\BG_1 \times \BG_2 := \{ ( (x_1,x_2),(z_1,z_2),(y_1,y_2) ) ~ | ~ (x_1,z_1,y_1)\in \BG_1 , (x_2,z_2,y_2)\in \BG_2  \}
\end{displaymath}
is a decomposition set of $P_1\times P_2$. 
\end{proposition}
\begin{proof}
Elementary calculations show, that triples of pairs, for which the triples of first and second coordinates are 
decompositions, are decompositions of the product poset. 
Since trivial decompositions of the product poset correspond to pairs of trivial decompositions of the factor posets, 
products of decomposition sets contain all trivial decompositions of the product. 
\end{proof}

\begin{proposition}
The product of maximal decomposition sets is the maximal decomposition set of the product poset.
\end{proposition}
\begin{proof}
Again, it is left to the reader to show, that decompositions of the product poset correspond to pairs of decompositions of
the initial posets. 
\end{proof}

\begin{proposition}
The product of minimal decomposition sets is the minimal decomposition set of the product poset if and only if 
the factors posets are anti chains. 
\end{proposition}
\begin{proof}
Pairs of trivial decompositions correspond to trivial decompositions of the product if and only if 
either the middle coordinates both equal their resp. first coordinates or they both equal their second coordinates. 
If and only if both factor posets are no anti chains this is not the case for any such pair. 
\end{proof}

\begin{theorem}
\label{theoprod}
For decomposition sets $\BG_1$ of $P_1$ and $\BG_2$ of $P_2$ and a chain $C$ in $P_1 \times P_2$ 
\begin{displaymath}
\langle C \rangle_{\BG_1 \times \BG_2} = \langle \pi_1 (C) \rangle_{\BG_1} \times \langle \pi_2 (C) \rangle_{\BG_2} 
\end{displaymath}
holds, where $\pi_1, \pi_2$ denote the projection maps of the product poset. 
\end{theorem}
\begin{proof}
Let us show the first inclusion by induction on the size of $\BG_1$ and $\BG_2$. 
For the start we consider $\Gmin(P_1)\times \Gmin(P_2)$. 
Decompositions, corresponding to pairs of trivial decompositions, indeed generate new elements to $C$, 
but they change neither $\pi_1(C)$ nor $\pi_2(C)$. 
Thus even $\langle C \rangle_{\Gmin(P_1) \times \Gmin(P_2)}=\pi_1(C)\times \pi_2(C)$ is true.  
For the induction itself set $\BG'_1:= \BG_1 \cup \{ (x_1,z_1,y_1)\}$. 
If $\langle C \rangle_{\BG'_1 \times \BG_2}= \langle C \rangle_{\BG_1 \times \BG_2}$, by the 
induction hypothesis, this equals $\langle \pi_1 (C) \rangle_{\BG_1} \times \langle \pi_2 (C) \rangle_{\BG_2}$, 
which is a subset of $\langle \pi_1 (C) \rangle_{\BG'_1} \times \langle \pi_2 (C) \rangle_{\BG_2}$. 
So assume $(z_1,z_2) \in \langle C \rangle_{\BG'_1 \times \BG_2} \backslash \langle C \rangle_{\BG_1 \times \BG_2}$ 
for some $z_2\in \langle \pi_2 (C) \rangle_{\BG_2}$. 
Then there has to be a decomposition 
$((x_1,x_2),(z_1,z_2),(y_1,y_2))\in \BG'_1 \times \BG_2 \backslash \BG_1 \times \BG_2$ with 
$(x_1,x_2),(y_1,y_2)\in \langle C \rangle_{\BG_1 \times \BG_2}=
 \langle \pi_1 (C) \rangle_{\BG_1} \times \langle \pi_2 (C) \rangle_{\BG_2}$. 
Thus in particular $x_1,y_1$ are contained in $\langle \pi_1 (C) \rangle_{\BG_1}$. 
This gives $z_1\in \langle \pi_1 (C) \rangle_{\BG'_1}$ and so $(z_1,z_2)\in 
 \langle \pi_1 (C) \rangle_{\BG'_1} \times \langle \pi_2 (C) \rangle_{\BG_2}$.

For the second inclusion, let $C=\{(c_i,d_i)|c_i \in P_1,d_i \in P_2, 1\leq i \leq n\}$ and 
$(a,b)\in \langle \pi_1 (C) \rangle_{\BG_1} \times \langle \pi_2 (C) \rangle_{\BG_2}$. 
Then $(c_i,d_j)\in \langle C \rangle_{\BG_1 \times \BG_2}$ for any $1\leq i,j \leq n$, because w.l.o.g. 
for $i \leq j$ the decomposition $((c_i,d_i),(c_i,d_j),(c_j,d_j))\in \BG_1 \times \BG_2$. 
Thus for any $1\leq j \leq n$, $(a,d_j)$ is in $\langle C \rangle_{\BG_1 \times \BG_2}$ since 
$a\in \langle \pi_1 (C) \rangle_{\BG_1}$. 
By the same argument $(a,b)$ is in $\langle C \rangle_{\BG_1 \times \BG_2}$, because 
$b\in \langle \pi_2 (C) \rangle_{\BG_2}$. 
\end{proof}

\begin{corollary}
Let $\BG_1,\BG_2$ be decomposition sets of $P_1$ resp. $P_2$, then 
%For decomposition sets $\BG_1$ of $P_1$ and $\BG_2$ of $P_2$, the map 
\begin{displaymath}
\mathcal{D}(P_1 \times P_2,\BG_1 \times \BG_2)\cong 
\mathcal{D}(P_1,\BG_1)\times \mathcal{D}(P_2,\BG_2)
\end{displaymath}
via the canonical isomorphisms of taking products and projecting to coordinates. 
\end{corollary}
\begin{proof}
Since the maps preserve order and do invert each other it is enough to check, that they are well defined. 
For chains $\{ c_1, \ldots , c_n\} \subseteq P_1$ and $\{ d_1, \ldots, d_m\}\subseteq P_2$, 
Theorem \ref{theoprod} applied to the chain $\{ (c_1,d_1),\ldots ,(c_n,d_1),\ldots,(c_n,d_m)\}$
immediately shows that taking products is well defined. 
By Theorem \ref{theoprod} we obtain 
$\pi_1(\langle C \rangle_{\BG_1 \times \BG_2})=\langle \pi_1(C) \rangle_{\BG_1}$ and 
$\pi_2(\langle C \rangle_{\BG_1 \times \BG_2})=\langle \pi_2(C) \rangle_{\BG_2}$
So the projection maps are well defined, too. 
For all decomposition complexes, the empty set has to be ignored, 
because for those special instances, taking products is not injective.
\end{proof}

\begin{example}
Let $P$ be the power set lattice of rank $3$. 
As in Example \ref{ex:second}, the identity is a realization, even for the maximal decomposition set. 
Then the realization of the decomposition complex resp. the minimal decomposition set is a triangulated 
unit cube in $\R^3$ into six $3$-simplices, all sharing the edge between $(0,0,0)$ and $(1,1,1)$. 
On the left of Figure \ref{fig:cubes} the realization resp. the minimal decomposition set is shown. 
On the right, we can see a realization resp. the maximal decomposition set. 
It is the unsubdivided unit cube. 

Since the power set lattice of rank $3$ is canonically isomorphic the product of the power set lattices 
$2^{\{1,2\}}$ and $2^{\{3\}}$, the product of the minimal decomposition sets of the factors gives a decomposition set, 
which is neither the minimal nor the maximal one. 
%It consists of all triples except for $(\emptyset,1,12)$, $(\emptyset,2,12)$,$(3,13,123)$ and $(3,23,123)$. 
The middle of Figure \ref{fig:cubes} shows the realization of the decomposition complex resp. this 
decomposition set. 

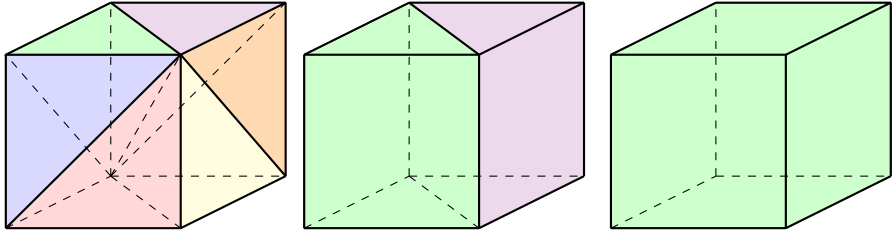
\begin{figure}[h]
{\small
\centering
\subfloat
{
\begin{tikzpicture}[scale=2.3]
 \coordinate (A1) at (0, 0);
    \coordinate (A2) at (0, 1);
    \coordinate (A3) at (1, 1);
    \coordinate (A4) at (1, 0);
    \coordinate (B1) at (0.6, 0.3);
    \coordinate (B2) at (0.6, 1.3);
    \coordinate (B3) at (1.6, 1.3);
    \coordinate (B4) at (1.6, 0.3);
   \filldraw[red!15] (A1)--(A4)--(A3)--(A1);
\filldraw[blue!15] (A1)--(A2)--(A3)--(A1);
\filldraw[yellow!15] (A4)--(B4)--(A3)--(A4);
\filldraw[fill=green!20]  (B2)--(A3)--(A2)--(B2);
\filldraw[violet!15] (A3)--(B3)--(B2)--(A3);
\filldraw[orange!30] (A3)--(B3)--(B4)--(A3);  
\draw[thick] (A1) -- (A2);
    \draw[thick] (A2) -- (A3);
    \draw[thick] (A3) -- (A4);
    \draw[thick](A4) -- (A1);
\draw[dashed] (B1) -- (A4);
\draw[dashed] (B1) -- (A2);
\draw[dashed] (B1) -- (A3);
\draw[dashed] (B1) -- (B3);
\draw[thick] (A3) -- (B2);
\draw[thick] (A3) -- (A1);
\draw[thick] (A3) -- (B4);
    \draw[dashed] (A1) -- (B1);
    \draw[dashed] (B1) -- (B2);
    \draw[thick](A2) -- (B2);
    \draw[thick](B2) -- (B3);
    \draw[thick](A3) -- (B3);
    \draw[thick](A4) -- (B4);
    \draw[thick](B4) -- (B3);
    \draw[dashed] (B1) -- (B4);
\end{tikzpicture}
}
{
\begin{tikzpicture}[scale=2.3]
 \coordinate (A1) at (0, 0);
    \coordinate (A2) at (0, 1);
    \coordinate (A3) at (1, 1);
    \coordinate (A4) at (1, 0);
    \coordinate (B1) at (0.6, 0.3);
    \coordinate (B2) at (0.6, 1.3);
    \coordinate (B3) at (1.6, 1.3);
    \coordinate (B4) at (1.6, 0.3);
\filldraw[fill=green!20]  (B2)--(A3)--(A4)--(A1)--(A2)--(B2);
\filldraw[violet!15] (A3)--(A4)--(B4)--(B3)--(B2)--(A3);
\draw[thick] (A1) -- (A2);
    \draw[thick] (A2) -- (A3);
    \draw[thick] (A3) -- (A4);
    \draw[thick](A4) -- (A1);
\draw[dashed] (B1) -- (A4);
\draw[thick] (A3) -- (B2);
    \draw[dashed] (A1) -- (B1);
    \draw[dashed] (B1) -- (B2);
    \draw[thick](A2) -- (B2);
    \draw[thick](B2) -- (B3);
    \draw[thick](A3) -- (B3);
    \draw[thick](A4) -- (B4);
    \draw[thick](B4) -- (B3);
    \draw[dashed] (B1) -- (B4);
\end{tikzpicture}
}
{
\begin{tikzpicture}[scale=2.3]
 \coordinate (A1) at (0, 0);
    \coordinate (A2) at (0, 1);
    \coordinate (A3) at (1, 1);
    \coordinate (A4) at (1, 0);
    \coordinate (B1) at (0.6, 0.3);
    \coordinate (B2) at (0.6, 1.3);
    \coordinate (B3) at (1.6, 1.3);
    \coordinate (B4) at (1.6, 0.3);
   \filldraw[fill=green!20]  (B2)--(B3)--(B4)--(A4)--(A1)--(A2)--(B2);
\draw[thick] (A1) -- (A2);
    \draw[thick] (A2) -- (A3);
    \draw[thick] (A3) -- (A4);
    \draw[thick](A4) -- (A1);
    \draw[dashed] (A1) -- (B1);
    \draw[dashed] (B1) -- (B2);
    \draw[thick](A2) -- (B2);
    \draw[thick](B2) -- (B3);
    \draw[thick](A3) -- (B3);
    \draw[thick](A4) -- (B4);
    \draw[thick](B4) -- (B3);
    \draw[dashed] (B1) -- (B4);
\end{tikzpicture}
}
}
\caption{Realizations of the power set lattice of rank $3$ resp. different decomposition sets}
\label{fig:cubes}
\end{figure}
\end{example}

\begin{theorem}
Let $\BG_1$ be a decomposition set of $P_1$ and let $\BG_2$ be a decomposition set of $P_2$. 
Then
\begin{displaymath}
D ( P_1 \coprod P_2, \BG_1 \cup \BG_2) ~=~  D(P_1,\BG_1) ~\coprod ~  D(P_2,\BG_2). 
\end{displaymath}
\end{theorem}
\begin{proof}
Since decompositions of the coproduct correspond to decompositions of the summands and chains 
are either contained in $P_1$ or in $P_2$, the statement is true, like in the case of order complexes. 
\end{proof}

\begin{remark}
\label{prodreal}
It is left to the reader to check, that the properties of symmetry and downwards closedness are preserved under 
taking products. The same is true for realizations. 
Point-wise products of $\BG-$ and $\BGG$-realizations clearly are $\BG \times \BGG-$realizations. 
\end{remark}

\section{Nested set complexes}
\label{ch5}
Let $\mathcal{P}$ be a poset with unique minimal element $\hat{0}$. 
For a subset $\BG$ of $P\backslash \{ \hat{0}\}$ we define $F_\BG (y)$ 
to be the set of maximal elements of $\BG_{\leq y}$. 

\begin{definition}
Then $\BG$ is called a \emph{building set} if for all $y \neq \hat{0}$ there is an isomorphism 
\begin{displaymath}
\psi_y : \Pi_{z_i \in F_\BG (y)}[\hat{0},z_i] \rightarrow [\hat{0},y],
\end{displaymath}
which is induced by the inclusions of intervals \emph{i.e.} $\psi_y( \hat{0},\ldots,\hat{0},x_i,\hat{0},\ldots,\hat{0} )
=x_i$. 

A set $S \subseteq \BG$ is called \emph{nested} if for any subset $A\subseteq S$ of size at least 2, 
the join $\bigvee A$ exists but is not in $\BG$ any more. 
The nested sets, ordered by inclusion form an abstract simplicial complex, called the \emph{nested set complex} 
of $P$ resp. $\BG$. 
\end{definition}

\begin{theorem}
For a poset $P$ with unique minimal element $\hat{0}$ and a building set $\BG$ of $P$, the map  
\begin{align*}
\phi: \mathcal{N}(P,\BG) &\rightarrow 
\mathcal{D}(P,\BG_{min}\cup \{ (\hat{0},z,y)\in \BG_{max}~ | ~ y \notin \BG \}) \\
S &\mapsto \{ \bigvee A | A\subseteq S\} 
\end{align*}
is an embedding with image $\{ \Cerz |~ C \text{ chain with } \hat{0}\in C\}.$
\end{theorem} 
\begin{proof}
First, note that all the joins in $ \{ \bigvee A | A\subseteq S\} $ are well defined, because $S$ is nested. 
For any such $S$, 
there is a one to one correspondence between linear extensions of the subposet $(S,\subseteq)$ 
and maximal chains 
in $ \{ \bigvee A | A\subseteq S\} $. 
To a linear ordering $a_1,\ldots ,a_n$, it assigns the chain 
$\{ \bigvee_{j\leq i}a_j | 0 \leq i \leq n \}$. 

Assume $a_1,\ldots, a_i,a_{i+1},\ldots ,a_n$ and $a_1,\ldots, a_{i+1},a_i,\ldots ,a_n$ are both linear extensions 
and $C_1,C_2$ are their assigned chains. 
In particular, $a_i$ and $a_{i+1}$ are incomparable and both are contained in $B:= F_\BG (\bigvee_{j\leq i+1}a_j) $. 
Obviously any two linear extensions can be transformed into each other by a series of such adjacent exchanges. 

The following calculations just use that all subsets of $B\subset S$ are nested sets.  
 \begin{align*}
[\hat{0},\bigvee B] &\cong \Pi_{a_j \in B} [\hat{0},a_j] \cong \Pi_{a_j \in B, j\neq i} [\hat{0},a_j] \times [\hat{0},a_i] \\
&\cong  [\hat{0},\bigvee B\backslash \{ a_i\}] \times [\hat{0},a_i]  \\
&\cong [\hat{0},\bigvee B\backslash \{ a_i\}] \times [\bigvee B\backslash \{ a_i\},\bigvee B]
\end{align*}
So $(\hat{0},\bigvee B\backslash \{ a_i\},\bigvee B)$ is a proper decomposition of the decomposition set 
$\BG_{min}\cup \{ (\hat{0},z,y)\in \BG_{max}~ | ~ y \notin \BG \}$, because by definition joins of nested sets 
must not belong to the building set. 
Since $\hat{0}$ as well as $\bigvee B=\bigvee_{j\leq i+1}a_j$ 
are both contained in $\{ \bigvee_{j\leq i}a_j | 0 \leq i \leq n \}$, 
we obtain that the single element of $C_2 \backslash C_1$, 
namely $\bigvee B \backslash \{ a_i \}$, belongs to $\langle C_1 \rangle$. 
Thus $\langle C_1 \rangle \subseteq \langle C_2 \rangle$ and 
therefore $\{ \bigvee A | A\subseteq S\}$ is generated by any of its maximal chains.
Thus the map is well defined, since order preserving is trivial and injectivity is due to the fact, 
that $x \in S$ if and only if $x \in \phi(S)$. 

For the statement about the image, note, that 
by definition the empty join is $\hat{0}$. So clearly all images contain $\hat{0}$. 
Since the operator $\langle \cdot \rangle$ cannot generate elements, which are lesser than the minimal element 
of the generating chain, this chain has to contain $\hat{0}$. 

On the other hand, consider a chain $C$ containing $\hat{0}$. 
Setting $S:= \bigcup_{c \in C} F_\BG(c)$, it is easy to show, that 
any maximal extension of $C $ in $\{ \bigvee A | A\subseteq S\}$ already generates the latter.  
\end{proof}

\section{Bergman Fans}
\label{ch6}
%Let $M$ be a  matroid with ground set $E(M)=[n]$. 
%We identify the matroid with its set of bases. 
%We start with a construction introduced by Feichtner/Yuzvinsky(\cite{FY04}).
Let $M$ be a finite, simple matroid, $\mathcal{L}_M$ its lattice of flats and 
$\{ a_1,\ldots, a_n \}$ the ground set of $M$, which are also the atoms of $\mathcal{L}_M$. 
\begin{proposition}
\label{prop:atom}
The (canonical) map 
\begin{displaymath}
\phi: \mathcal{L}_M \rightarrow B_{n},~ x \mapsto \{ i \in [n] | a_i \leq x \} 
\end{displaymath}
is a $\BG_{max}$-realization. 
\end{proposition}
\begin{proof}
The map above just uses the interpretation of flats as subsets of the ground set. 
%In the lattice of flats, the meet operation is just the usual 
Let $(x,z,y),(x,z',y)$ be a pair of complementary decompositions.
Then surely $x=z \wedge z'$ and $y= z \vee z'$ are true. 
Now $\phi(x)=\phi(z)\cap \phi(z')$ holds, because the set theoretic intersection of flats 
is the meet operation in the lattice of flats. 
Because of the decomposition $[x,y,]\cong[x,z]\times[x,z']$, 
we obtain that $\phi(z) \cup \phi(z')$ is already the flat $\phi(z\vee z')$.  
%Thus $\phi(x)=\phi(z)\cap \phi(z')$ and $\phi(y)=\phi(z)\cup \phi(z')$ hold. 
\end{proof}

\begin{definition}
\label{incid2}
%For $A\subseteq E(M)$, its incidence vector $e_A\in \{ 0,1\}^n$ is defined by $(e_A)_i=1 $ if and only if $a_i \subseteq A$.  
The \emph{matroid polytope} of $M$ is the convex hull of the incidence vectors of bases of $M$ in $\R^{n}$. 
\end{definition} 
This is a pure polytope of dimension $n-c(M)$, where $c(M)$ denotes the number of connected components of $M$. 
Its subfaces are matroid polytopes of matroids, which are direct sums of minors of $M$, called matroid types, themselves. 
Bases of matroid types correspond one-to-one to vertices of the subface. 
These bases are the possible outputs of the greedy algorithm resp. some weight vector $\omega$. 
Thus matroid types are denoted by $M_\omega$. 
The weight vectors $\omega$, which induce the same matroid type, form a cone in $\R^n$. 
Those cones are invariant under translations of the orthogonal complement of $P_M$ in $\R^n$, which is
the subspace generated by the incidence vectors of seperators of $M$. 
The set of those cones form a complete polyhedral fan $\widetilde{{P_M}^*}$. 
\begin{definition}
The \emph{Bergman fan} $\widetilde{\mathcal{B}}(M)$ is the subfan of $\widetilde{{P_M}^*}$ 
consisting of the cones whose induced matroid type is loopfree. 
\end{definition}
Because of the invariances of the cones, we loose no information when restricting the dimension in the following way. 
\begin{definition}
The \emph{Bergman complex} $\mathcal{B}(M)$ is the intersection of $\widetilde{\mathcal{B}}(M)$ with 
the unit-sphere $S^{n-1}$ and the linear span of $P_M$. 
\end{definition}
This gives a spherical, polyhedral complex, whose face poset is the same as the face poset of the Bergman fan, which 
is isomorphic to the poset of loopfree matroid types ordered by reversed inclusion of bases. 

\begin{theorem}
\label{thm:bergman}
Let $M$ be a finite, loopfree matroid, $\mathcal{L}_M$ its lattice of flats, 
$B(M)$ its Bergman fan and $\mathcal{F}(B(M))$ its face poset. 
The map
\begin{displaymath}
\psi: \mathcal{F}(B(M)) \rightarrow D(\mathcal{L}_M,\Gmax)
, ~ \Mw \mapsto \{ A \in \mathcal{L}_M | \rank(A)=|A\cap b|~  \forall b \in \mathcal{B}(\Mw) \} 
\end{displaymath}
is an embedding.
A subset $A$ is in the image iff it contains $\hat{0}$ and $\hat{1}$. 
\end{theorem}
\begin{proof}
The fact, that $\langle \psi(M_\omega) \rangle=\psi(M_\omega)$, \emph{i.e.} the closedness 
under taking connected components, is proven in \cite[Prop. 5.3]{Dlu11}. 
So it is left to show that $\psi(M_\omega)$ is generated by a single chain. 
The maximal decomposition set is both symmetric as well as downwards closed. 
So in view of Remark \ref{partofa}, it suffices to show, that any pair of elements 
can be generated from a single chain. 
For $z,z'\in \psi(M_\omega)$, from \cite[Prop.5.3]{Dlu11} follows, that 
$\rank(z)+\rank(z')=\rank(z\wedge z')+\rank(z \vee z')$. 
Thus the rank of the matroid $(M| z) /x\oplus (M| z') /x$ equals the rank of $(M |y) /x$.
Therefore there are complementary decompositions 
$(x,z,y),(x,z',y)\in \Gmax$. 
Again by \cite[Prop. 5.3]{Dlu11}, we obtain that $x=z\wedge  z',y=z\vee z' \in \psi(M_\omega)$. 
So the chain $\{x,y\}\subseteq \psi(M_\omega)$ generates both $z$ and $z'$. 

Order preserving is satisfied, because the bigger a face of the Bergman complex, the fewer bases has the 
corresponding matroid type $M_\omega$, the fewer restrictions for flats there are to satisfy, the more 
flats do satisfy those. 
Injectivity holds, because $M_\omega$ is exactly the set of bases $b$ satisfying $|b\cap x|=\rank(x)$ for all 
$x\in \psi(M_\omega)$. 

Left to show is the determination of the image of the embedding. 
Since for all bases $b\in M$, $|b\cap x|=\rank(x)$ holds for $x\in \{\hat{0},\hat{1}\}$, only such images can occur. 
On the other hand, for any chain $C$ containing $\hat{0},\hat{1}$, we obtain that 
$\langle C \rangle=\psi(M_\omega)$ for $\omega=\sum_{c \in C}e_c$. 
\end{proof}

\begin{proposition}
%For the realization $\phi$ of Proposition \ref{prop:atom}, 
The Bergman fan is essentially the 
decomposition fan resp. the realization $\phi$ of Proposition \ref{prop:atom}, \emph{i.e.}
\begin{displaymath}
\widetilde{\mathcal{B}}(M)= \R_{\geq 0} ( \widetilde{D}(\mathcal{L}_M,\Gmax)) + \R (1,\ldots,1),
\end{displaymath}
where $\R_{\geq 0}$ stands for non-negative scaling. 
\end{proposition}
\begin{proof}
Any vector $\omega\in \R^n$ can be written as 
$ \sum_{i=1}^k \lambda_i e_{F_i} + \mu (1,\ldots,1)$, where $\mu \in \R$, $\lambda_i >0$, 
$ F_{i+1} \subseteq F_i$ and $e_{F_i}$ is the incidence vector of $F_i$. 
The sets $F_1,\ldots,F_n$ in this presentation are uniquely determined, though the coefficients are not. 
In \cite[Theorem 1]{AK04} it is shown, that the induced matroid type $M_\omega$ is loopfree if and only if 
$F_1,\ldots, F_n$ are flats of $M$. 
But in view of Corollary \ref{cor:conedec}, vectors of the form $ \sum_{i=1}^k \lambda_i e_{F_i}$ 
with $\lambda_i \geq 0$ 
are exactly the ones lying inside 
$\R_{\geq 0} ( \widetilde{D}(\mathcal{L}_M,\Gmax))$.
\end{proof}

\bibliography{bibgeneral}{}
\bibliographystyle{amsalpha}

\end{document}